\documentclass[11pt]{amsart}

\usepackage[dvipsnames]{xcolor}
\usepackage{latexsym,amssymb,amsmath,hyperref,amsthm,amsfonts,caption,subcaption,tikz,comment}
\usepackage{mdwlist,dirtytalk,times,cancel}
\usepackage{newtxmath}
\usepackage[capitalise, noabbrev]{cleveref}
\usetikzlibrary{intersections, arrows.meta, automata,er,calc, backgrounds, mindmap,folding, patterns, decorations.markings, fit,decorations, matrix, positioning, shapes.geometric, arrows,through, graphs, graphs.standard}

\usepackage{amsrefs}
\usepackage{enumitem}

\numberwithin{equation}{section}
\usetikzlibrary{positioning}
\usetikzlibrary{arrows}

\newtheorem{theorem}{Theorem}[section]

\newtheorem{proposition}[theorem]{Proposition}
\newtheorem{corollary}[theorem]{Corollary}
\newtheorem{conjecture}[theorem]{Conjecture}

\theoremstyle{definition}
\newtheorem{definition}[theorem]{Definition} 
\newtheorem{remark}[theorem]{Remark}
\newtheorem{example}[theorem]{Example}

\DeclareMathOperator{\depth}{depth}

\DeclareMathOperator{\Ind}{Ind}

\DeclareMathOperator{\Facets}{Facets}

\newcommand{\qand}{\quad \mbox{and} \quad}

\newcommand{\qfor}{\quad \mbox{for} \quad}

\newcommand{\F}{\mathcal{F}}
\newcommand{\N}{\mathcal{N}}
\renewcommand{\P}{{\mathcal{P}}}

\newcommand{\st}{\colon}
\newcommand{\lk}{\mathrm{link}}
\newcommand{\del}{\mathrm{del}}

\begin{document}
 
\title{Spheres and balls as independence complexes}

\author[S.~M.~Cooper]{Susan M. Cooper}
\address[S. M. Cooper]
{Department of Mathematics\\
University of Manitoba,
Winnipeg, MB\\
Canada R3T 2M6}
\email{susan.cooper@umanitoba.ca}

\author[S. Faridi]{Sara Faridi}
\address[S. Faridi]
{Department of Mathematics \& Statistics,
Dalhousie University,
6297 Castine Way,
PO BOX 15000,
Halifax, NS,
Canada B3H 4R2
}
\email{faridi@dal.ca}

\author[T. Holleben]{Thiago Holleben}
\address[T. Holleben]
{Department of Mathematics \& Statistics,
Dalhousie University,
6297 Castine Way,
PO BOX 15000,
Halifax, NS,
Canada B3H 4R2}
\email{hollebenthiago@dal.ca}

\author[L. Nicklasson]{Lisa Nicklasson}
\address[L. Nicklasson]
{Division of Mathematics and Physics,
Mälardalen University,
Universitetsplan 1,
722 18 Västerås, Sweden}
\email{lisa.nicklasson@mdu.se}

\author[A. Van Tuyl]{Adam Van Tuyl}
\address[A. Van Tuyl]
{Department of Mathematics and Statistics
McMaster University, Hamilton, ON, L8S 4L8}
\email{vantuyla@mcmaster.ca}

\keywords{simplicial complexes, simplicial spheres, Bier spheres, pseudomanifold, Cohen-Macaulay, very well-covered graphs, grafting,  monomial ideals, Artinian rings, polarization}
\subjclass[2020]{13F55, 05E45, 05C69}

 
\begin{abstract}
The terms \say{whiskering}, and more generally \say{grafting}, refer to adding generators to any monomial 
ideal to make the resulting ideal  Cohen-Macaulay.
We investigate the independence complexes
of simplicial complexes that are constructed through
a whiskering or grafting process, 
and we show that these independence complexes are (generalized) Bier balls. More specifically, the independence complexes are either homeomorphic
to a ball or sphere.  In a related direction, we 
classify when
the independence complexes of very well-covered graphs are homeomorphic to balls or spheres.
\end{abstract}
\maketitle

\section{Introduction}

    In 1990, Villarreal~\cite[Proposition 2.2]{V1990} introduced an operation that takes any graph $G$ and outputs a graph $w(G)$ with a Cohen-Macaulay edge ideal. Over the years, this construction was extensively studied in commutative algebra and became known as \say{whiskering}. In unpublished work from 1992,
    and brought to a wider attention by \cites{Ma2003, BPS2005, dL2004}, Bier ~\cite{B1992} 
    introduced a simple operation that inputs any simplicial complex $\Delta$ on $n$ vertices, and outputs a simplicial complex $B(\Delta)$ homeomorphic to a $(n-1)$--dimensional sphere. The complex $B(\Delta)$ is obtained by taking the boundary of a simplicial ball $\Delta'$ and, when $\Delta'$ is flag, its Stanley-Reisner ideal is exactly the one introduced by Villarreal in~\cite{V1990}. Bier's construction is a strikingly simple way of generating many (combinatorially) different spheres~\cite[Subsection 6.1]{BPS2005}. As a consequence of this 
    observation, the realizability of these spheres (and their generalizations) as the boundary of a (simplicial) polytope became an interesting problem that has been studied from different perspectives, see for example~\cites{JTZ2021,BN2012}.

    From an algebraic perspective, Villarreal's construction is intrinsically related to polarizations of monomial ideals. In view of this, Faridi in 2005~\cite{Fa2005} introduced the notion of \say{grafted complexes}, which are essentially the simplicial complexes whose facet ideals are polarizations of artinian monomial ideals. Similar to whiskering, Faridi's construction takes any (square-free) monomial ideal and outputs a Cohen-Macaulay ideal generated by square-free monomials. In 2011, Murai~\cite{M2011} generalized Bier's construction by considering simplicial complexes arising from multicomplexes. Our first contribution in this paper is to show that Murai's construction coincides with Faridi's construction, in the same sense that Bier's construction coincides with Villarreal's. The statement below summarizes \cref{thm.mainthm-new}.
    
    \begin{theorem}[{\bf Generalized Bier balls and grafted complexes}]\label{t:main1}
        A simplicial complex is grafted if and only if 
        its independence
        complex is a generalized Bier ball $B(M)$, where $M$ is a multicomplex.
    \end{theorem}

     As a consequence of~\cref{t:main1}, Faridi's results from~\cite{Fa2005} can be seen as a characterization of the minimal nonfaces of generalized Bier balls from~\cite{M2011}.

     The ideal  connecting grafted complexes and generalized Bier balls  in \cref{t:main1} is the polarization of an artinian monomial ideal.
    We note that more general notions of polarizations were considered in~\cite[Conjecture 3.2]{AFL2022}, where the authors conjecture that any notion of polarization of an artinian monomial ideal leads to a simplicial ball. A proof of this conjecture was recently announced in~\cite{FGM}.

    In commutative algebra, Faridi's construction is only one of many generalizations of whiskering. Our second contribution shows that the notion of coloured-whiskering from Biermann and Van Tuyl~\cite{BVT2013} almost never outputs a pseudomanifold.
    In particular, these complexes are a very large class of interesting vertex decomposable but non-pseudomanifold simplicial complexes.
    
    \begin{theorem}
        Let $\Delta$ be a simplicial complex and $\Delta_\chi$ be the complex obtained via coloured-whiskering with respect to a partition $\chi$ of the vertices. Then the following are equivalent:
    
    \begin{enumerate}
            \item $\Delta_\chi$ is a pseudomanifold;
            \item $\Delta_\chi$ is a grafted complex.
        \end{enumerate}
    \end{theorem}

    From a completely different perspective, in 1982 Favaron~\cite{F1982} introduced the class of very well-covered graphs. This class of graphs includes whiskered graphs, which are sometimes called \say{leafy extensions}~\cite{BC2018}  in this context.
     We are able to make the following conclusions about the independence complex of $G$ when $G$ is a very well-covered graph. Combining with Stanley's result \cite[Theorem II.5.1]{Stanley1996} this classifies Gorenstein very well-covered graphs.

    \begin{theorem}\label{t:main3}
        Let $G$ be a very well-covered graph with edge ideal $I$.  If $G$ is not a disjoint
        set of edges, then $I$ is Cohen-Macaulay if and only if the Stanley-Reisner complex of
        $I$, also known as the
        independence complex of $G$, is homeomorphic to a ball. 
        If $G$ is a disjoint set of edges, then $I$ is Gorenstein and the independence complex
        of $G$ is homeomorphic to a sphere.
        In particular, the edge ideal of a very well-covered graph $G$ is Gorenstein if and only if $G$ is a disjoint set of edges.
    \end{theorem}
  We note that a great deal of the research on very well-covered graphs in graph theory is focused on the unimodality of their independence polynomials~\cites{LM2006, BC2018}. For the sake of completeness, we state Levit and Mandrescu's conjecture in the context of flag simplicial complexes.

  \begin{conjecture}[\cite{LM2006}]\label{c:1}
      If $\Delta$ is a $(d-1)$--dimensional pure flag complex on $2d$ vertices, its $f$-vector is unimodal. 
  \end{conjecture}

\cref{t:main3} suggests the study of~\cref{c:1} in the special case where $\Delta$ is Cohen-Macaulay (for example when $\Delta$ is a flag Bier ball) is of special interest.

\section{Background}

We begin with a review of necessary background 
on simplicial complexes and 
commutative algebra.

\subsection{Simplicial Complexes}
A (abstract) {\bf simplicial complex} $\Delta$ on a {\bf vertex set} 
$V = V(\Delta)$ is a set of subsets of $V$ with the condition 
that if $F \in \Delta$ and $G \subseteq F$, then $G \in \Delta$.  
An element of $\Delta$ is a called a {\bf face} of $\Delta$, and 
 a {\bf facet} of $\Delta$ is a maximal face under inclusion. The {\bf dimension} of a face $F$ is $\dim(F) = |F|-1$,
the {\bf dimension} of $\Delta$ is the maximum of the dimensions of its faces,
and $\Delta$ is {\bf pure} if all its facets have the same dimension.
A  {\bf simplex} is a simplicial complex with exactly one  facet.
We often denote $\Delta$ as  $\langle F_1, \ldots, F_q \rangle$ meaning  $F_1, \ldots, F_q$ are the facets of $\Delta$ (in particular, a simplicial complex is  uniquely determined by its facets).
The {\bf link} and {\bf deletion} of a face $F$ of $\Delta$ are, respectively,   the subcomplexes  
$$\lk_\Delta(F) = \{G \in \Delta \st G \cup F \in \Delta,\  G \cap F  = \emptyset\} 
\qand  
\del_\Delta(F) = \{ G \in \Delta \st  G \cap F = \emptyset\}.$$
A subset $W \subseteq V$ is an {\bf independent set} of 
$\Delta$ if it contains no  facet of $\Delta$, and 
$\Ind(\Delta)$, called the {\bf independence complex} of $\Delta$, 
is the simplicial complex whose faces are the independent sets of $\Delta$; in other words,  
$$ 
\Ind(\Delta)= \{W \subseteq V \st W  \text{ is an independent set of }\Delta\}.
$$

For every simplicial complex $\Delta$, we may associate a topological space $|\Delta|$ contained in $\mathbb{R}^k$ for some $k$ such that for every face $\sigma \in \Delta$, $|\Delta|$ contains the convex hull of the corresponding vertices of $\sigma$ (see~\cite[Chapter 1]{Ma2003} for more details on this construction).
We call $|\Delta|$ the {\bf geometric realization} of $\Delta$.
When the geometric realization of a $d$-dimensional
simplicial complex $\Delta$ is homeomorphic to a sphere $S^d$, we say $\Delta$ is a {\bf simplicial sphere}, and when it  is homeomorphic to a ball $B^d$, we say $\Delta$ is a {\bf simplicial ball}. 

In this paper we are interested in 
identifying simplicial spheres and 
simplicial balls. One way to determine this property is to check
if the complex is a shellable  (or even more strongly, vertex decomposable) pseudomanifold.

Our main context is the \say{Cohen-Macaulay} property in algebra (\cref{d:CM}), which forces simplicial complexes to be pure. Therefore, we will  focus on pure simplicial complexes for the next definitions.

\begin{definition}[{\bf Pseudomanifolds, shellable, vertex decomposable}]\label{def.stronglyconnect}
Let $\Delta$ be a pure $d$-dimensional simplicial complex.

\begin{itemize}
\item $\Delta$ is {\bf shellable}  if its facets can be
ordered as $F_1,\ldots,F_q$ such that  for all $i=2,\ldots,q$ 
$$\langle F_i \rangle \cap \langle F_1,\ldots,F_{i-1} \rangle$$ 
is  pure and of dimension $d-1$.

\item $\Delta$ is {\bf vertex decomposable}
if 
\begin{enumerate}
    \item $\Delta = \emptyset$, or
    $\Delta$ is a simplex, or 
    \item there exists a vertex $v \in V$
    such that  every facet of 
    $\del_\Delta(\{v\})$ is a facet of
    $\Delta$, and both
     $\lk_\Delta(\{v\})$ and
    $\del_\Delta(\{v\})$ are vertex 
    decomposable.
\end{enumerate}

\item  A face of dimension $(d-1)$ in $\Delta$ that belongs to at most two facets is called a {\bf ridge}.  
\item $\Delta$ is said to be {\bf strongly-connected} if for every two facets $F,F' \in \Delta$,
    there exists a sequence of facets 
    $$ F=G_0,G_1,\ldots,G_t=F' 
    \quad \mbox{such that} \quad  
    \dim (G_i \cap G_{i+1}) = d-1$$ 
    for $i = 0,\ldots,t-1$.

\item $\Delta$ is a 
{\bf pseudomanifold} if
\begin{enumerate}  
    \item $\Delta$ is pure;
    \item every face of dimension $(d-1)$ is a ridge; and 
    \item $\Delta$ is strongly-connected.
\end{enumerate}

\item a pseudomanifold $\Delta$ is said to have a {\bf boundary} if
there exists at least one face of dimension $(d-1)$
of $\Delta$ that belongs to exactly one facet of
$\Delta$.   
    
\end{itemize}
\end{definition}

The series of implications in \eqref{implications} shows how these classes of simplicial complexes relate to one another. In the sections that follow, we will rely  on the following statement to detect if a simplicial complex is a simplicial sphere or ball.

\begin{theorem}[{\cite[Theorem 11.4]{Bj1995}}]\label{t:sphere-ball}
Let $\Delta$ be a shellable pseudomanifold. Then
\begin{itemize}
 \item if $\Delta$ has no boundary, then  it is a simplicial sphere; 
 \item if $\Delta$ has a boundary, then  it is a simplicial ball.
\end{itemize}
\end{theorem}

\subsection{Square-free monomial ideals}
There are two one-to-one correspondences between simplicial complexes and square-free 
monomial ideals. More precisely, to a simplicial complex $\Delta$ with vertex set $\{x_1,\ldots,x_n\}$ we associate two unique 
square-free monomial ideals: $\F(\Delta)$, called the {\bf facet ideal} of $\Delta$, and $\N(\Delta)$,  called the {\bf Stanley-Reisner} ideal of $\Delta$, in the polynomial ring  $k[x_1,\ldots,x_n]$ over a field $k$, defined as
$$
\F(\Delta)=\big (x_{i_1}\cdots x_{i_r} \st \{x_{i_1},\ldots,x_{i_r}\} \in \Facets(\Delta) \big ) 
 \qand 
 \N(\Delta)=\big (x_{i_1}\cdots x_{i_r} \st \{x_{i_1},\ldots,x_{i_r}\} \notin \Delta \big ).$$
We use $x_1, \ldots, x_n$ to denote both vertices and variables in the polynomial ring. 
The uniqueness of the correspondences allow us to define $\F(I)$ and $\N(I)$ for a square-free monomial ideal $I$ by 
$$I=\F(\Delta) \iff  \Delta=\F(I) \qand I=\N(\Delta) \iff \Delta=\N(I).$$

It is straightforward to check from the definitions that if $I$ is a square-free monomial ideal, then
\begin{equation}\label{e:connectfacetSR}
\N(I)=\Ind(\F(I)).
\end{equation}

\begin{example}\label{e:facet-SR}
If $I=(abc,cd,ad) \subset k[a,b,c,d]$, then $\Delta=\F(I)$ and 
$\Gamma=\N(I)$ are pictured below. Note that the facets of $\Gamma$ are exactly the maximal independent sets of $\Delta$, that is, $\Gamma=\Ind(\Delta)$.

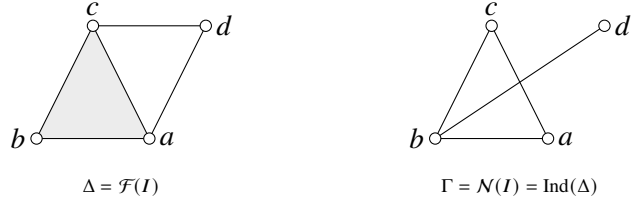
\begin{center}
\begin{tabular}{cc}
\begin{tikzpicture}[scale=0.75]
 \coordinate (A) at (2,0);
 \coordinate (B) at (0,0);
 \coordinate (C) at (1,2);
 \coordinate (D) at (3,2);
  \draw [fill=gray!15] (A.center) -- (B.center) -- (C.center)--(A.center);
  \draw (A) -- (D);
  \draw (C) -- (D);
\foreach \point in {A, B, C, D} \fill[fill=white,draw=black] (\point) circle (.1);
\node[right] at (A) {$a$};
\node[left] at (B) {$b$};
\node[above] at (C) {$c$};
\node[right] at (D) {$d$};
\end{tikzpicture}
 $\qquad$ & $\qquad$ 
\begin{tikzpicture}[scale=0.75]
 \coordinate (A) at (2,0);
 \coordinate (B) at (0,0);
 \coordinate (C) at (1,2);
 \coordinate (D) at (3,2);
  \draw (A.center) -- (B.center) -- (C.center)--(A.center);
  \draw (B) -- (D);
\foreach \point in {A, B, C, D} \fill[fill=white,draw=black] (\point) circle (.1); 
\node[right] at (A) {$a$};
\node[left] at (B) {$b$};
\node[above] at (C) {$c$};
\node[right] at (D) {$d$};
\end{tikzpicture}\\
{\tiny $\Delta=\F(I)$}  $\qquad$ & $\qquad$ {\tiny $\Gamma=\N(I)=\Ind(\Delta)$}
\end{tabular}
\end{center}
\end{example}

\begin{example}[{\bf Edge ideals}]  Let $G = (V,E)$ be a finite simple
graph with vertices $V = \{x_1,\ldots,x_n\}$ and
edge set $E$.  If we view $G$ as a simplicial complex with 
its edges the facets, then $\mathcal{F}(G) = 
( x_ix_j \st \{x_i,\,x_j\} \in E )$
in
$k[x_1,\ldots,x_n]$. The ideal $\mathcal{F}(G)$ is in fact the well-known  {\bf edge ideal} of $G$, 
and is normally denoted by $I(G)$.  The Stanley-Reisner
complex associated to $I(G)$ is the 
{\bf independence complex} $\Ind(G)$
of $G$.
Simplicial complexes that are the independence complex of some graph are also called {\bf flag complexes}.
\end{example}

\subsection{The Cohen-Macaulay property}
The Cohen-Macaulay property -- defined algebraically for a local ring $R$ via the property that $\depth(R)=\dim(R)$ --  is the dominating algebraic property of the square-free monomial ideals that we study in this paper. For quotients by square-free monomial ideals, the Cohen-Macaulay property was characterized  by Reisner~\cite{Reisner} in terms of the vanishing of reduced homology groups of the links of faces of Stanley-Reisner complexes.  Reisner's criterion has been used to great effect in combinatorics ever since its appearance, and remains the most concrete description of Cohen-Macaulayness for the class of square-free monomial ideals. In \cref{d:CM} we will use Reisner's criterion as the definition of Cohen-Macaulay square-free monomial ideals.

\begin{definition}[{\bf Cohen-Macaulay ideals (Reisner's Criterion)}] \label{d:CM}
Let $I$ be a square-free monomial ideal
of $R=k[x_1,\ldots,x_n]$ with non-face
complex $\Delta = \mathcal{N}(I)$.  Then
$R/I$ is a {\bf Cohen-Macaulay ring} if and only if
for any $F \in \Delta$, 
$$
\tilde{H}_i(\lk_\Delta(F),k) = 0
\qfor 
i < \dim (\lk_\Delta(F)).$$ 
We call $\Delta$ a {\bf Cohen-Macaulay
simplicial complex} if $R/\mathcal{N}(\Delta)$ 
is a Cohen-Macaulay ring.
\end{definition}

The simplest example of a simplicial complex that satisfies the conditions in \cref{d:CM} is a simplex, where all homologies of links are zero.
A more interesting basic example is a hollow simplex, where all the link homologies are one-dimensional.
The hollow simplex is 
an example of what is known as a \say{homology sphere}. For a field $k$, a  simplicial $\Delta$  is a {\bf $k$-homology sphere} if for all $F \in \Delta$
\[
\tilde{H}_i(\lk_\Delta(F),k) 
    = \begin{cases} k & \mbox{if $i = 
    \dim (\lk_\Delta(F))$} \\
    0 & \mbox{otherwise.}
     \end{cases}
\]
Simplicial spheres, introduced earlier,  are examples of homology spheres and are therefore Cohen-Macaulay ({\cite[Theorem 5.1, Corollary 5.2]{Stanley1996}}). It is known that every Cohen-Macaulay 
simplicial complex is strongly-connected~(\cite[Proposition 11.7]{Bj1995}, \cite[Proposition 1.2.12]{A2017}).

The various classes of simplicial complexes
introduced in this section satisfy
the implications
\vspace*{5pt}
\begin{equation}\label{implications}
\begin{array}{ccccccc}
\mbox{simplicial sphere}  & \Rightarrow &
\mbox{homology sphere} & \Rightarrow & 
\mbox{pseudomanifold} & \Rightarrow & 
\mbox{pure} \\
& & & \Searrow &  &\Searrow & \Uparrow \\
\mbox{vertex decomposable} & \Rightarrow &
\mbox{shellable} & \Rightarrow  & 
\mbox{Cohen--Macaulay} & \Rightarrow & 
\mbox{strongly-connected}\\
\end{array}.
\end{equation}

All of the reverse implications in \eqref{implications}
fail to hold in general.  However, as we 
will show in the sequel, 
implications can be reversed for special families of independence complexes.

We conclude this section with a description of the \say{polarization} of a monomial ideal, which is an operation that transforms a general monomial ideal into a square-free one, while maintaining its Cohen-Macaulay properties. If $m=x_1^{a_1}\cdots x_n^{a_n} \in k[x_1,\ldots,x_n]$ is a monomial, then its {\bf polarization} $\P(m)$ is the square-free monomial
$$\P(m)= (x_{1,1}x_{1,2}\cdots x_{1,a_1}) \cdots (x_{n,1}\cdots x_{n,a_n}) \in k[x_{i,j} \st 1 \leq i \leq n, \ 1  \leq j \leq a_i].$$
If $I$ is minimally generated by monomials $m_1,\ldots,m_q$, then the {\bf polarization of} $I$ is the square-free monomial ideal
$$\P(I)=\big( \P(m_1),\ldots,\P(m_q) \big )$$ 
in the polynomial ring containing all variables appearing in the monomials $\P(m_1),\ldots,\P(m_q)$.

\section{Grafting, whiskering, spheres, and generalized Bier balls}

 When the Cohen-Macaulay property is present in  a monomial ideal, both its facet and Stanley-Reisner complexes demonstrate strong combinatorial characteristics. However, despite Reisner's criterion, it is quite difficult to determine if a ring is Cohen-Macaulay, because the process involves computing homology groups, which itself could depend on the characteristic of the base field. A fruitful line of research, therefore, has been to identify classes of monomial ideals $I$ for which we can tell, based on the combinatorics of $\N(I)$ or $\F(I)$, that $I$ is Cohen-Macaulay no matter what field we use. 
 For example, when $\N(I)$ is a simplicial sphere, 
 then $I$ must be Cohen-Macaulay.
 Similarly, if $\F(I)$ is \say{grafted} $I$ must be Cohen-Macaulay.  Grafting~\cite{Fa2005}  is essentially adding generators to a square-free monomial ideal in a way that the extended ideal is Cohen-Macaulay.
This is a generalization of {\it whiskering} a graph, first discovered by
Villarreal \cite{V1990}. We begin by recalling how to graft a simplicial
complex.

For a simplicial complex $\Delta$ with vertex set $V$, a {\bf leaf} $F$ is a facet where either $F$ is the only facet of $\Delta$, or there is another facet $G$ of $\Delta$   such that 
$F \cap H \subseteq G$ for all facets $H \neq F$.  The facet $G$ is called a {\bf joint} of $F$ (see \cites{Fa2002, Fa2005}).

The operation of grafting is done via adding leaves to a simplicial complex.
In the statement below, $V(\Delta')$ refers to the vertices
of $\Delta'$, and $V(F)$ refers to the vertices of the facet $F$.

\begin{definition}[{\bf Grafting~\cite{Fa2005}}]\label{d:grafting}
A simplicial complex $\Delta$ is a {\bf grafting} of the simplicial complex $\Delta' = \langle G_1, \ldots, G_s \rangle$ with the simplices $F_1, \ldots, F_r$ (or, $\Delta$ is said to be {\bf grafted}) if 
$$\Delta = \langle F_1, \ldots, F_r \rangle \cup \langle G_1, \ldots, G_s \rangle$$
with the following properties:
\begin{enumerate}
\item $V(\Delta') \subseteq V(F_1) \cup \cdots \cup V(F_r)$;
\item $F_1, \ldots, F_r$ are all the leaves of $\Delta$;
\item $\{G_1, \ldots, G_s\} \cap \{F_1, \ldots, F_r\} = \emptyset$;
\item for $i \not = j$, $F_i \cap F_j = \emptyset$;
\item if $G_i$ is a joint of $\Delta$, then $\Delta \setminus \langle G_i \rangle$ is also grafted.
\end{enumerate}
In the special case where the leaves $F_1,\ldots,F_r$ are edges (i.e., $1$-dimensional), then $\Delta$ is called a {\bf whiskered}
simplicial complex.
\end{definition}

\begin{example} The facet complex of  
 $I_2=(x_1x_2x_3, x_1x_2x_4,x_1x_3x_4, x_1x_3y_1,  x_4y_2, x_2y_3)$ seen below, is a grafting of the facet complex of the ideal $I_1=(x_1x_2x_3, x_1x_2x_4,x_1x_3x_4)$.
By \cite[Theorem 7.6]{Fa2005} the ideal $I_2$  is Cohen-Macaulay, obtained by adding generators (leaves) to $I_1$.
$$
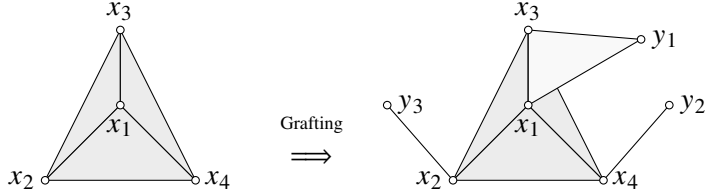
\begin{tabular}{ccc}
\begin{tikzpicture}[scale=0.5]
 \coordinate (A) at (0,0);
 \coordinate (AA) at (0,-.1);
 \coordinate (B) at (-2,-2);
 \coordinate (C) at (0,2);
 \coordinate (D) at (2,-2);

  \draw [fill=gray!15] (A.center) -- (B.center) -- (C.center)--(A.center);
  \draw [fill=gray!15] (A.center) -- (B.center) -- (D.center)--(A.center);
  \draw [fill=gray!15] (A.center) -- (D.center) -- (C.center)--(A.center);
 
\foreach \point in {A, B, C, D} \fill[fill=white,draw=black] (\point) circle (.1);
 \node[below] at (AA) {$x_1$};
 \node[left] at (B) {$x_2$};
 \node[above] at (C) {$x_3$};
 \node[right] at (D) {$x_4$};

\end{tikzpicture}
&
\begin{tabular}{c}
{\tiny Grafting}\\
$\Longrightarrow$\\
\\ \\ \\
\end{tabular}
&
 \begin{tikzpicture}[scale=0.5]
 \coordinate (A) at (0,0);
  \coordinate (AA) at (0,-.1);
    \coordinate (B) at (-2,-2);
    \coordinate (C) at (0,2);
    \coordinate (D) at (2,-2);
    \coordinate (E) at (3,1.75);
    \coordinate (F) at (3.75,0);
    \coordinate (G) at (-3.75,0);
 
  \draw [fill=gray!15] (A.center) -- (B.center) -- (C.center)--(A.center);
  \draw [fill=gray!15] (A.center) -- (B.center) -- (D.center)--(A.center);
  \draw [fill=gray!15] (A.center) -- (D.center) -- (C.center)--(A.center);
  \draw [fill=gray!05] (E.center) -- (A.center) -- (C.center)--(E.center);
  \draw [-] (D.center)-- (F.center);
  \draw [-] (B.center)-- (G.center);
\foreach \point in {A, B, C, D, E, F, G} \fill[fill=white,draw=black] (\point) circle (.1);
 \node[below] at (AA) {$x_1$};
 \node[left] at (B) {$x_2$};
 \node[above] at (C) {$x_3$};
 \node[right] at (D) {$x_4$};
 \node[right] at (E) {$y_1$};
 \node[right] at (F) {$y_2$};
 \node[right] at (G) {$y_3$};
\end{tikzpicture}
\end{tabular}
$$
 \end{example}
 
\begin{remark}[{\bf Whiskering a graph}~\cite{V1990}]
    In 1990, Villarreal introduced the whiskering construction for the edge ideal of a graph, citing Vasconcelos and Herzog's observations that this process served as a way to obtain a Cohen-Macaulay square-free monomial ideal $J = w(I)$ from any square-free monomial ideal $I$, by adding quadratic monomials to the ideal. Then in 2005, Faridi introduced the notion of grafted simplicial complexes to generalize Villarreal's construction, in the sense that higher degree monomials could be added to any monomial ideal to produce a Cohen-Macaulay ideal.  Both of these constructions were introduced in a different language by Bier in an unpublished manuscript from 1992, and Murai in~\cite{M2011}. The notion of a whiskered graph specifically has also appeared in the graph theory literature under several names, for example the \say{leafy-extension} of a graph~\cite{BC2018} and the \say{corona product} of a graph with a single vertex~\cite{FH1970}.  
\end{remark}

As we will see below in \cref{thm.mainthm-new}, the independence complex of a grafted complex is exactly what is known as a {\bf generalized Bier ball}, which was  defined by Murai~\cite{M2011} as the Stanley-Reisner complex of the polarization of an artinian monomial ideal.
Bier spheres first appeared in unpublished work of Thomas Bier in 1992 
according to~\cite{M2011}. 
This construction was then generalized by Murai in~\cite{M2011}, where it is shown that the independence complex of the polarization of any artinian monomial ideal is a simplicial ball or sphere. The original case due to Bier is the case of an artinian monomial ideal that contains the square of every variable, or, by  \cref{thm.mainthm-new}, the independence complex of a whiskered simplicial complex.  Generalized Bier balls are always balls, except for one special case where they turn out to be spheres (as shown below).

 \begin{theorem}[{\bf Independence complexes of grafted complexes}]\label{thm.mainthm-new}
 Let $\Delta$ be a simplicial complex, let $\Gamma = \Ind(\Delta)$
  be the independence complex of $\Delta$,  and let $I=\F(\Delta)=\N(\Gamma) \subseteq k[x_1,\ldots,x_n]$. The following are equivalent.
\begin{itemize}
    \item $\Delta$ is grafted;
    \item $\Gamma$ is a generalized Bier ball;
    \item $I$ is the polarization of an artinian monomial ideal.
\end{itemize}
 If $I$ is the polarization of an ideal minimally generated by monomials $\{x_1^{a_1},\ldots,x_n^{a_n}\} \cup A$, then 
 $\Gamma$ is vertex decomposable and $\dim(\Gamma)=a_1+\cdots+a_n -n- 1$. Moreover, 
 $\Gamma$ is a simplicial sphere if and only if $A=\emptyset$, and is otherwise a simplicial ball.
\end{theorem}

\begin{proof} 
The proof of Theorem~8.2 in \cite{Fa2005} shows that the facet ideal of a grafted simplicial complex is the polarization of an artinian monomial ideal, which by Murai~\cite[Lemma~3.2]{M2011}  is the Stanley-Reisner ideal of a generalized Bier ball. 
The remaining statements are \cite[Lemma~1.4,~Remark~1.8]{M2011}. See also~\cite[Theorem 2.16]{MRV2008}.
\end{proof}

\begin{example} We will highlight the cases where we get a ball or a sphere.
Let $I_1= (ab,bc,cd)$ and $I_2=(ab,cd)$ be polarizations
of the artinian monomial ideals $(x^2,xy,y^2)$ and 
$(x^2,y^2)$, respectively. The facet and independence complexes of $I_1$ and $I_2$ are then the following. 

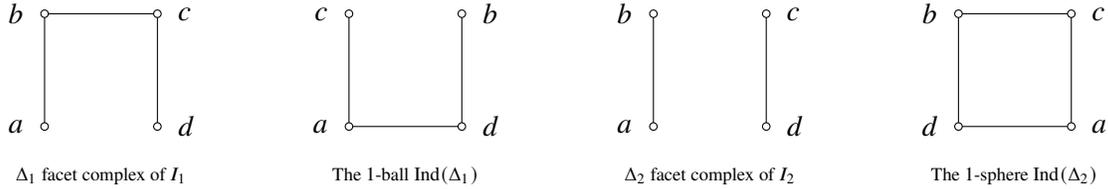
\begin{center}
\begin{tabular}{cccc}
\begin{tikzpicture}[scale=0.5]
\draw (0,0) -- (0,3) -- (3,3) -- (3,0); 
\fill[fill=white,draw=black] (0,0) circle (.1)
node[label=left:$a$] {};
\fill[fill=white,draw=black] (0,3) circle (.1)
node[label=left:$b$] {};
\fill[fill=white,draw=black] (3,3) circle (.1)
 node[label=right:$c$] {};
\fill[fill=white,draw=black] (3,0) circle (.1)
node[label=right:$d$] {};
\end{tikzpicture}
$\quad$& $\quad$
\begin{tikzpicture}[scale=0.5]
\draw (0,3) -- (0,0) -- (3,0) -- (3,3); 
\fill[fill=white,draw=black] (0,0) circle (.1)
node[label=left:$a$] {};
\fill[fill=white,draw=black] (0,3) circle (.1)
node[label=left:$c$] {};
\fill[fill=white,draw=black] (3,3) circle (.1)
 node[label=right:$b$] {};
\fill[fill=white,draw=black] (3,0) circle (.1)
node[label=right:$d$] {};
\end{tikzpicture}
$\quad$& $\quad$
\begin{tikzpicture}[scale=0.5]
\draw (0,0) -- (0,3);
\draw (3,3) -- (3,0); 
\fill[fill=white,draw=black] (0,0) circle (.1)
node[label=left:$a$] {};
\fill[fill=white,draw=black] (0,3) circle (.1)
node[label=left:$b$] {};
\fill[fill=white,draw=black] (3,3) circle (.1)
 node[label=right:$c$] {};
\fill[fill=white,draw=black] (3,0) circle (.1)
node[label=right:$d$] {};
\end{tikzpicture}
$\quad$& $\quad$
\begin{tikzpicture}[scale=0.5]
\draw (0,0) -- (0,3) -- (3,3) -- (3,0) -- (0,0); 
\fill[fill=white,draw=black] (0,0) circle (.1)
node[label=left:$d$] {};
\fill[fill=white,draw=black] (0,3) circle (.1)
node[label=left:$b$] {};
\fill[fill=white,draw=black] (3,3) circle (.1)
 node[label=right:$c$] {};
\fill[fill=white,draw=black] (3,0) circle (.1)
node[label=right:$a$] {};
\end{tikzpicture}\\
{\tiny $\Delta_1$ facet complex of $I_1$}
$\quad$& $\quad$
{\tiny The $1$-ball $\Ind(\Delta_1)$} 
$\quad$& $\quad$
{\tiny $\Delta_2$ facet complex of $I_2$}
$\quad$& $\quad$
{\tiny The $1$-sphere $\Ind(\Delta_2)$}
\end{tabular}
\end{center}
\end{example}

A byproduct of \cref{thm.mainthm-new}  is new conditions added to the list of equivalent conditions for facet ideals of simplicial trees (or edge ideals of graphs) to be Cohen-Macaulay. The first three equivalences in \cref{cor.trees} were proved for graphs in~\cite{V1990} and generalized to simplicial complexes in~\cite{Fa2005}. The new items are consequences of \cref{thm.mainthm-new}.

\begin{corollary}[{\bf Cohen-Macaulay (simplicial) trees}]\label{cor.trees}
If $\Delta$ is a simplicial tree (and in particular a graph which is a tree),  then the following are equivalent.
\begin{itemize}
    \item $\Delta$ is unmixed; 
\item $R/\F(\Delta)$ is Cohen-Macaulay;
\item $\Delta$ is grafted;
\item $\Ind(\Delta)$ is a generalized Bier ball;
\item $\Ind(\Delta)$ is vertex decomposable. 
\end{itemize}
\end{corollary} 

\section{Simplical balls and spheres via general whiskering}

The previous section showed that 
grafting a simplicial complex, which was defined in 2005  as a generalization of Villarreal's graph whiskering from 1990,  results in a simplicial complex whose independence complex is a simplicial
ball or sphere. 
In 2012 Cook and Nagel \cite{CN2012} introduced another generalization of Villarreal's graph-whiskering called the  
\say{clique-whiskering} of graphs. To clique-whisker a graph, 
the vertex set $V$ is partitioned into 
$V = V_1 \cup \cdots \cup V_s$ 
such that the induced graph on each $V_i$ is a clique, that is, a subset of the vertices
that are all adjacent to each other.
For each clique, a new vertex is added to the graph, 
and this new vertex is joined to all vertices in the clique.  Whiskered graphs are then the case that the graph
is partitioned into cliques of single vertices.  
The independence complexes of clique-whiskered graphs
enjoy many similar properties to 
independence complexes of whiskered graphs, e.g.,
they are pure and vertex decomposable, and in particular, Cohen-Macaulay.

Biermann and Van Tuyl~\cite{BVT2013} in 2013 generalized
Cook and Nagel's construction to all simplicial complexes.
In vew of the earlier discussions, it is natural to ask if any of these Cohen-Macaulay \say{generalized whiskered complexes} of Biermann and Van Tuyl are 
simplicial balls or spheres.
It turns out that this is rarely the case, and when it is, the generalized whiskered complex is in fact the independence complex of a grafted (or more specifically, whiskered) complex.  

We begin by recalling the construction of a generalized whiskered complex from~\cite{BVT2013}.   
Let $\Delta$ be a simplical
complex on the vertex set $V$ 
with facets $F_1,\ldots,F_t$.  
An {\bf $s$-colouring} $\chi$ is a partition $V = V_1 \cup V_2 \cup \cdots \cup V_s$ 
such that for each  $|V_i \cap F| \leq 1$ for all 
$i=1,\ldots,s$ and every facet $F$.  We do allow some of the $V_i$ to be empty.

\begin{example}[{\bf The $n$-colouring}]\label{ex.ncolour}
    Suppose that $\Delta$ is a simplicial complex on
    $V = \{x_1,\ldots,x_n\}$. Then
    $V = \{x_1\}\cup \{x_2\} \cup \cdots \cup \{x_n\}$
    is an $n$-colouring of $\Delta$ since
    $|\{x_i\} \cap F| \leq 1$ for all facets 
    $F \in \Delta$.
\end{example}

From a simplicial complex $\Delta$ and an  $s$-colouring $\chi$, we construct another simplicial complex $\Delta_{\chi}$ on vertex set $V \cup \{y_1, \ldots, y_s\}$. For a face $F$ of $\Delta$, let 
$$Y_F
=\{y_j \st F \cap V_j = \emptyset\} 
\subseteq \{y_1,\ldots,y_s\}.$$
Then the faces of $\Delta_{\chi}$ are the sets $F \cup G$ such that $F$ is a face of $\Delta$ and $G \subseteq Y_F$. In particular
    \begin{equation}\label{defn.generalwhisker}
    \Delta_\chi = 
    \left \langle 
    F \cup Y_F \st 
    F \in \Delta
   \right \rangle.
\end{equation}

\begin{example}\label{e:run1}
Consider the simplicial complex
on $\{x_1,x_2,x_3,x_4\}$ given by
$\Delta  = \langle \{x_1,x_2,x_3\}, \{x_3,x_4\} \rangle$.  Consider the colouring
$\chi = \{x_1,x_4\} \cup \{x_2\} \cup 
\{x_3\}$ below, represented by the different shaded vertices.
$$
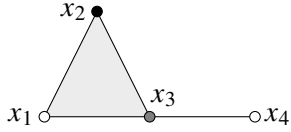
\begin{tikzpicture}[scale=0.7]
    \coordinate (X1) at (-1,-1);
    \coordinate (X2) at (0,1);
    \coordinate (X3) at (1,-1);
    \coordinate (X4) at (3,-1);
  \draw (X3) -- (X4);
  \draw [fill=gray!15] (X1.center) -- (X2.center) -- (X3.center)--(X1.center);
\fill[fill=white,draw=black] (X4) circle (.1);
\node[right] at (X4) {$x_4$};
\fill[fill=white,draw=black] (X1) circle (.1);
 \node[left] at (X1) {$x_1$};
 \fill[fill=black,draw=black] (X2) circle (.1);
 \node[left] at (X2) {$x_2$};
 \fill[fill=gray,draw=black] (X3) circle (.1);
 \node[above] at (X3) {$\quad x_3$};
 \end{tikzpicture}
$$
Then the facets of $\Delta_\chi$,
which is a simplicial complex on
$\{x_1, x_2, x_3, x_4,y_1,y_2,y_3\}$, 
are 
$$\begin{array}{ccccc}
 \{x_1,x_2,x_3\} &  \{x_1,x_2,y_3\}
& \{x_1,y_2,x_3\} & \{y_1,x_2,x_3\} & 
\{x_3,x_4,y_2\} \\
\{x_1,y_2,y_3\} & \{y_1,x_2,y_3\} & 
\{y_1,y_2,x_3\} & \{y_2,y_3,x_4\} 
& \{y_1,y_2,y_3\}.
\end{array}$$

\end{example}

The following theorem recalls some properties
of $\Delta_\chi$.

\begin{theorem}[{\cite[Theorems 3.5, 3.7]{BVT2013}}]\label{thm.whiskeredcomplexprop}
Let $\Delta$ be a simplicial
complex on $V =\{x_1,\ldots,x_n\}$. 
Suppose that the partition $V=V_1\cup \cdots \cup V_s$ is an $s$-colouring $\chi$ of $\Delta$. Then

\begin{enumerate}
    \item $\dim \Delta_\chi=s-1$;
    \item $\Delta_\chi$ is vertex
    decomposable, and thus shellable, Cohen--Macaulay, and pure.
\end{enumerate}
\end{theorem}

When $\chi$ is the
$n$-colouring described in  \Cref{ex.ncolour}, we show that  
$\Delta_\chi$ is exactly  the independence complex of a whiskered simplicial complex~(\cref{d:grafting}, see also~{\cite[Remark 3.16]{BVT2013}}).

\begin{proposition}\label{lem:coloredwhisker} \label{cor.facetsncolouring}
Suppose that $\Delta$ is a simplicial
complex on $V =\{x_1,\ldots,x_n\}$, and
$\chi$ is the $n$-colouring 
$\{x_1\}\cup \cdots \cup \{x_n\}$.
Then 
\begin{enumerate}
\item every facet $G \in \Delta_\chi$ has the form
$$G = \{z_1,\ldots,z_n\} ~~\mbox{where $z_i \in \{x_i,y_i\}$~
for $i=1,\ldots,n$}.$$

\item If $\Delta=\Ind(\Gamma)$ for some simplicial complex $\Gamma$, then $\Delta_\chi$ is the independence complex of a  whiskering of $\Gamma$.
In other words, 
$$\N(\Delta_{\chi}) = \N(\Delta) + 
( x_1 y_1, \dots, x_n y_n ).$$
\end{enumerate}
\end{proposition}

\begin{proof}
Suppose $G=F\cup Y_F$ is a facet of $\Delta_\chi$ for some $F \in \Delta$. Then $Y_F=\{y_j \st 
 x_j \notin F \}$ and so we cannot have
both $x_i$ and $y_i$ in $G$.
By \cref{thm.whiskeredcomplexprop} $\Delta_{\chi}$ is a pure simplicial complex of dimension $n-1$, so $|G|=n$. Hence, for each $i$ either $x_i \in G$ or $y_i \in G$, but not both. This settles the first statement.

 For the second statement, note that if $F=\emptyset \in \Delta$, then $F \cup Y_F=\{y_1,\ldots,y_n\} \in \Delta_\chi$. Therefore, every nonface of $\Delta_\chi$ must contain at least one vertex $x_i$.  Part~(1) shows that for each $i \in \{1,\ldots,n\}$, $x_iy_i$ is a minimal nonface of $\Delta_\chi$. The only remaining minimal nonfaces are those containing only $x_i$'s, and those are exactly the minimal nonfaces of $\Delta$. Statement~(2) now follows from \eqref{e:connectfacetSR}.
 \end{proof}

\begin{example}\label{e:run2}
The independence complex of $\Gamma =\langle \{x_1,x_4\},\{x_2,x_4\},\{x_3\}\rangle$ is the simplicial complex
$\Delta$  in \cref{e:run1}. The $4$-colouring
$\chi = \{x_1\} \cup \{x_2\} \cup \{x_3\} \cup 
\{x_4\}$, produces the simplicial complex $\Delta_\chi$ on vertex set
$\{x_1,x_2, x_3, x_4,y_1,y_2, y_3,y_4\}$  with facets
$$\begin{array}{ccccc}
 \{y_1,y_2,y_3,y_4\} &
\{x_1,y_2,y_3,y_4\} & \{y_1,x_2,y_3,y_4\} & 
\{y_1,y_2,x_3,y_4\} & \{y_1,y_2,y_3,x_4\} 
\\
\{x_1,x_2,y_3,y_4\} &
\{x_1,y_2,x_3,y_4\} &  \{y_1,x_2,x_3,y_4\} & 
\{y_1,y_2,x_3,x_4\} &  \{x_1,x_2,x_3,y_4\} 
\end{array}$$
which is the independence complex of a whiskering of the  complex $\Gamma$,  drawn below on the left.
On the other hand, when $\chi$ is the $3$-colouring in \cref{e:run1}, then
$$\N(\Delta_\chi)=(x_1x_4,\, x_2x_4,\, x_1y_1,\, x_4y_1,\, x_2y_2,\, x_3y_3)$$ 
and $\Delta_\chi$ is not 
the independence complex of a grafted complex, as can be seen on the right.

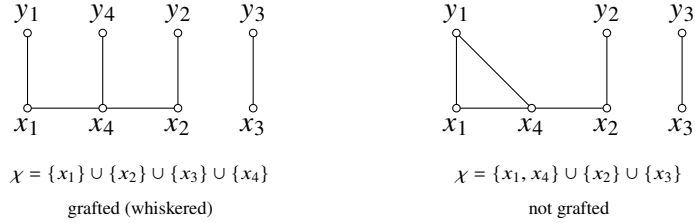
\begin{center}
\begin{tabular}{cc}
\begin{tikzpicture}[scale=0.5]
    \coordinate (X1) at (-1,-1);
    \coordinate (X3) at (5,-1);
    \coordinate (X4) at (1,-1);
    \coordinate (X2) at (3,-1);
    \coordinate (Y1) at (-1,1);
    \coordinate (Y3) at (5,1);
    \coordinate (Y4) at (1,1);
    \coordinate (Y2) at (3,1);
  \draw (X1) -- (X4);
  \draw (X2) -- (X4);   
  \draw (X1) -- (Y1); 
  \draw (X2) -- (Y2); 
  \draw (X3) -- (Y3); 
  \draw (X4) -- (Y4); 
\fill[fill=white,draw=black] (X1) circle (.1);
 \node[below] at (X1) {$x_1$};
 \fill[fill=white,draw=black] (X2) circle (.1);
 \node[below] at (X2) {$x_2$};
 \fill[fill=white,draw=black] (X3) circle (.1);
 \node[below] at (X3) {$x_3$};
\fill[fill=white,draw=black] (X4) circle (.1);
\node[below] at (X4) {$x_4$};
\fill[fill=white,draw=black] (Y1) circle (.1);
 \node[above] at (Y1) {$y_1$};
\fill[fill=white,draw=black] (Y2) circle (.1);
 \node[above] at (Y2) {$y_2$};
\fill[fill=white,draw=black] (Y3) circle (.1);
 \node[above] at (Y3) {$y_3$};
\fill[fill=white,draw=black] (Y4) circle (.1);
 \node[above] at (Y4) {$y_4$}; 
\end{tikzpicture}
$\qquad$&$\qquad$
\begin{tikzpicture}[scale=0.5]
    \coordinate (X1) at (-1,-1);
    \coordinate (X3) at (5,-1);
    \coordinate (X4) at (1,-1);
    \coordinate (X2) at (3,-1);
    \coordinate (Y1) at (-1,1);
    \coordinate (Y3) at (5,1);
    \coordinate (Y2) at (3,1);
  \draw (X1) -- (X4);
  \draw (X2) -- (X4);   
  \draw (X1) -- (Y1); 
  \draw (X2) -- (Y2); 
  \draw (X3) -- (Y3); 
  \draw (X4) -- (Y1); 
\fill[fill=white,draw=black] (X1) circle (.1);
 \node[below] at (X1) {$x_1$};
 \fill[fill=white,draw=black] (X2) circle (.1);
 \node[below] at (X2) {$x_2$};
 \fill[fill=white,draw=black] (X3) circle (.1);
 \node[below] at (X3) {$x_3$};
\fill[fill=white,draw=black] (X4) circle (.1);
\node[below] at (X4) {$x_4$};
\fill[fill=white,draw=black] (Y1) circle (.1);
 \node[above] at (Y1) {$y_1$};
\fill[fill=white,draw=black] (Y2) circle (.1);
 \node[above] at (Y2) {$y_2$};
\fill[fill=white,draw=black] (Y3) circle (.1);
 \node[above] at (Y3) {$y_3$};
\end{tikzpicture}
\\
{\tiny $\chi=\{x_1\} \cup \{x_2\} \cup \{x_3\} \cup \{x_4\}$}
$\qquad$&$\qquad$
{\tiny $\chi=\{x_1,x_4\} \cup \{x_2\} \cup \{x_3\}$}\\
{\tiny grafted (whiskered)}
$\qquad$&$\qquad$
{\tiny not grafted}
\end{tabular}
\end{center}
\end{example}

To classify which colourings $\chi$ make
$\Delta_\chi$ a sphere or a ball, we first show that it would be necessary for $\chi$ to be an $n$-colouring. 

\begin{theorem}\label{pro.coloredwhiskering} \label{cor.coloredwhiskering} \label{t.pseudomanifolds+boundary}
   Let $\Delta$ be a simplicial complex on the vertex set $V = \{x_1, \dots, x_n\}$ and $\chi$ a colouring of $\Delta$. The following are equivalent:
  \begin{enumerate}
      \item $\Delta_{\chi}$ is a pseudomanifold;
      \item $\chi$ is the $n$-colouring $V = \{x_1\} \cup \dots \cup \{x_n\}$;
      \item $\Delta_\chi$ is the independence complex of a whiskered complex;
      \item $\Delta_\chi$ is the independence complex of a grafted 
      complex;
      \item $\Delta_\chi$ is a simplicial sphere
  if $\Delta$ is a simplex, and a simplicial
  ball otherwise.  
   \end{enumerate}

\end{theorem}

\begin{proof}
 $(2) \Rightarrow (1)$
 We first suppose that the $s$-colouring
 $\chi$ is $V = \{x_1\} \cup \cdots \cup 
  \{x_n\}$.  Then $\Delta_\chi$ is 
  Cohen-Macaulay of dimension $n-1$ by   \cref{thm.whiskeredcomplexprop}, and thus $\Delta_\chi$
  strongly-connected (see~\eqref{implications}).
  Consider any face $F \in \Delta_\chi$ of dimension
  $n-2$.   Suppose that $G \in \Delta_\chi$ is a facet
  that contains $F$.  By \cref{cor.facetsncolouring},
  $G = \{z_1,z_2,\ldots,z_n\}$ with $z_i \in \{x_i,y_i\}$.
  Thus, there is some $j \in \{1\,\ldots,n\}$ such that
  $F = G \setminus \{z_j\}$. So, the only facets of $\Delta_\chi$
  that can contain $F$ are $F \cup \{x_j\}$ or $F \cup \{y_j\}$.
  Thus, each $n-2$ dimensional face of $\Delta_\chi$ is a
  ridge.  Consequently, $\Delta_\chi$ is a pseudomanifold.

  $(1) \Rightarrow (2)$ We prove the contrapositve
  statement. So suppose
  that $\chi$ is an $s$-colouring of
  $\Delta$ such that the associated
  partition $V = V_1 \cup \cdots \cup V_s$
  has at least one $V_i$ with $|V_i| \geq 2$.  
  After relabelling, we can
  assume $|V_1| \geq 2$, and $u,v \in V_1$ are two
  distinct vertices in $V_1$.   Let $\{y_1,\ldots,y_s\}$
  denote the new vertices of $\Delta_\chi$.  
  Then $\{y_1,\ldots,y_s\}$ is a facet of $\Delta_\chi$ by
  \cref{defn.generalwhisker} when
  we take $F = \emptyset$.  So, $\{y_2,\ldots,y_s\}$ 
  is a $(s-2)$-dimensional face that is contained in three
  facets, namely, $\{y_1,y_2,\ldots,y_s\}$, $\{u,y_2,
  \ldots,y_s\}$ and $\{v,y_2,\ldots,y_s\}$.  Thus,
  $\Delta_\chi$ cannot be a pseudomanifold since this face
  is not a ridge.

  For the other implications, $(2) \Rightarrow (3)$ follows from~\cref{lem:coloredwhisker}, $(3) \Rightarrow (4)$ follows by definition and $(4) \Rightarrow (1)$ follows from \cref{thm.mainthm-new}.
  
    For~(5), since simplicial spheres and balls are pseudomanifolds (see \eqref{implications}), we can assume $\chi$ 
  is the colouring $\{x_1\} \cup \cdots
  \cup \{x_n\}$ and, by what we just proved and \cref{lem:coloredwhisker}~(2), $\Delta_\chi$ is the independence complex of a whiskering of  $\Gamma$ where $\Ind(\Gamma)=\Delta$. By \cref{thm.mainthm-new}, $\Delta_\chi$ is a simplicial ball, unless $\Gamma=\emptyset$, or equivalently, $\Delta$ is the full simplex.
  \end{proof}

\begin{example} With the $3$-colouring $\chi$ in \cref{e:run1}, in light of 
\cref{pro.coloredwhiskering},  $\Delta_\chi$ is not 
a pseudomanifold, and thus not a sphere, since the one-dimensional
face $\{y_2,y_3\}$ appears in three facets.
However, $\Delta_\chi$ is Cohen--Macaulay
by \Cref{thm.whiskeredcomplexprop}.
On the other hand with the $4$-colouring of the same simplicial complex $\Delta$ in \cref{e:run2},  we obtain a complex $\Delta_\chi$ which is a ball since it is the independence complex of a whiskered simplicial complex that is not a simplex. 
\end{example}

\section{Simplicial balls and spheres as flag complexes of very well-covered graphs}\label{s:verywellcovered}

The implications in \eqref{implications}  summarized
how  one property of a simplicial complex leads to another one.
While the reverse implications fail to hold in general, in this
section we show that all the implications  (except
the pure property) can be reversed for independence
complexes of very well-covered graphs.  Some of these reverse implications were discovered earlier in~\cites{CRT2011, MMCRTY2011}. Our main contribution is to show that
the independence complex of a Cohen-Macaulay very well-covered graph with no isolated vertices is always a simplicial ball or sphere.
As a by-product of our work, we conclude that assuming the size of a facet of a pure flag complex is half the number of vertices -- a restriction often considered in graph theory -- is enough to make several interesting properties, such as being Cohen-Macaulay and being a simplicial ball be equivalent.

We begin by recalling some relevant graph
theory.
Let $G = (V,E)$ be a graph with vertex 
set $V$
and edge set $E$.  We sometimes
write $V(G)$ or $E(G)$ to highlight the vertex set
or edge set belonging to $G$.
If $W \subseteq V$, then the {\bf induced graph of $G$ on 
$W$} is the graph $G_W = (W,E_W)$ where 
$E_W = \{e \in E \st e \subseteq W\}$.  In other
words, $G_W$ is the subgraph of $G$ where an
edge $e$ of $G$ is an edge of $G_W$ if both
endpoints of $e$ belong to $W$.  A {\bf cycle} on
$n$ vertices, denoted $C_n$, 
is the graph with vertex set
$V =\{x_1,\ldots,x_n\}$ and edge set
$E = \{\{x_i,x_{i+1}\} \st i = 1,\ldots,n-1\} \cup
\{\{x_n,x_1\}\}$.  A {\bf clique} 
on $n$ vertices, denoted $K_n$, is 
the graph with vertices $V = \{x_1,\ldots,x_n\}$
and edge set $\{\{x_i,x_j\} \st 1 \leq i < j \leq n\}$.
In particular, $K_1$ denotes a vertex.
A subset $W \subseteq V$ is an {\bf independent set} 
of $G$ if
for all $e \in E$, $e \not\subseteq W$. 
{\bf Maximal independent sets} are independent
sets that are maximal with respect to inclusion.
A vertex cover of $G$ is a notion dual 
to an independent set.
Specifically, a subset $Y \subseteq V$ is 
a {\bf vertex cover} of $G$ if for all $e \in E$,
$e \cap Y \neq \emptyset$.   A vertex cover $Y$ 
is {\bf minimal} if
$Y$ is a vertex cover, but $Y \setminus\{y\}$ 
is not a vertex cover
for all $y \in Y$.   It is not hard to 
verify that $W \subseteq V$
is an (maximal) independent set of $G$ 
if and only if $V\setminus W$ is a (minimal) vertex
cover of $G$.  
A graph $G$ is {\bf unmixed} or {\bf well-covered} if all
of its maximal independent sets have 
the same cardinality.  Note this is equivalent to the
statement that all the minimal vertex covers have the
same size.

As noted earlier, the 
{\bf edge ideal} $I(G)$ of a graph $G$ is the 
facet ideal of $G$ when considered as a 
$1$-dimensional simplicial complex. 
Similarly, the well-known {\bf independence complex} of $G$ 
is the complex $\Ind(G)$ as defined earlier for simplicial complexes. In particular, the facets
of $\Ind(G)$ correspond to maximal independent sets,
and $\Ind(G)$ is pure if and only if 
$G$ is well-covered.

A graph $G$ is {\bf very well-covered} if
every maximal independent set has cardinality $|V|/2$.  Note that this
forces $|V|$ to be even.  Very well-covered graphs, which were first defined \cite{F1982},
have appeared often within
the combinatorial commutative algebra literature
(see the survey \cite{KPSFTY2022} and the references therein).
Whiskered graphs are examples of very well-covered graphs,
as we now show.

\begin{example}\label{ex.whisker-verywellcovered}
 Let $G$ be any graph with vertex set
 $V(G) = \{x_1,\ldots,x_d\}$
 and let $H = w(G)$ be the whiskered graph with
 $V(H) = \{x_1,\ldots,x_d,y_1,\ldots,y_d\}$. We
 verify that $H$ is very well-covered by showing
 that every maximal independent set has $d = |V(H)|/2$
 elements. Let $W$ be an independent set of $H$. First note that as
 $\{x_i,y_i\}$ is an edge of $H$, at most one of $x_i$ and $y_i$ can be in $W$ for each $i=1,\ldots,d$, and thus 
 $|W| \leq d$.  If $|W| <d$,
 then there is an $i$ such that $x_i, y_i \not\in W$, and $W \cup \{y_i\}$ is an independent set of $H$
 since $y_i$ is only adjacent to $x_i$. So if $W$ is maximal then $|W|=d$,
 and thus any whiskered graph is very well-covered.
 \end{example}

The converse, however, is not true. 
The graph
$$
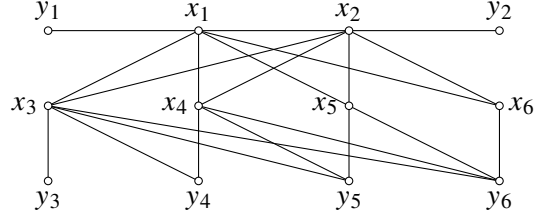
\begin{tikzpicture}[scale=.5]
\coordinate  (y3) at (0,0);
\coordinate  (y4) at (4,0);
\coordinate  (y5) at (8,0);
\coordinate  (y6) at (12,0);
\coordinate  (x3) at (0,2);
\coordinate  (x4) at (4,2);
\coordinate  (x5) at (8,2);
\coordinate  (x6) at (12,2);
\coordinate  (x1) at (4,4);
\coordinate  (x2) at (8,4);
\coordinate  (y1) at (0,4);
\coordinate  (y2) at (12,4);
\draw (y3) -- (x3);
\draw (y4) -- (x4);
\draw (y5) -- (x5);
\draw (y6) -- (x6);
\draw (y4) -- (x3);
\draw (y5) -- (x3);
\draw (y6) -- (x3);
\draw (y5) -- (x4);
\draw (y6) -- (x4);
\draw (y6) -- (x5);
\draw (x1) -- (x2);
\draw (x1) -- (x3);
\draw (x1) -- (x4);
\draw (x1) -- (x5);
\draw (x1) -- (x6);
\draw (x2) -- (x3);
\draw (x2) -- (x4);
\draw (x2) -- (x5);
\draw (x2) -- (x6);
\draw (y1) -- (x1);
\draw (x2)-- (y2);
\foreach \point in {x1,x2,x3,x4,x5,x6,y1,y2,y3,y4,y5,y6} \fill[fill=white,draw=black] (\point) circle (.1);
\node[below] at (y3) {${ y_3}$};
\node[below] at (y4) {$y_4$};
\node[below] at (y5) {$ y_5$};
\node[below] at (y6) {$ y_6$};
\node[left]  at (x3) {$ x_3$};
\node[left]  at (x4) {$ x_4$};
\node[left]  at (x5) {$ x_5$};
\node[right] at (x6) {$ x_6$};
\node[above] at (x1) {$ x_1$};
\node[above] at (x2) {$ x_2$};
\node[above] at (y1) {$ y_1$};
\node[above] at (y2) {$ y_2$};
\end{tikzpicture}
$$
is Cohen-Macaulay and very well-covered by~\cite{CRT2011} (see \Cref{thm.verywellcovered-label} below), but not whiskered.

 As indicated earlier, some of the reverse implications 
 of \eqref{implications} are known to hold for the
 independence complexes of very well-covered graphs.  We 
recall the specific statements below.

\begin{theorem}[{\cite[Theorem 0.2]{CRT2011}},
{\cite[Theorem 1.1]{MMCRTY2011}}, {\cite[Theorem~2.3]{CV}}]
\label{thm.verywellcoveredclass} 
Let $G$ be a very well-covered graph without
isolated vertices.  Then the following 
conditions are equivalent:
\begin{enumerate}
    \item $\Ind(G)$ is strongly-connected;
    \item $I(G)$ is Cohen-Macaulay;
    \item $\Ind(G)$ is shellable; 
    \item $\Ind(G)$ is vertex-decomposable.
\end{enumerate}
\end{theorem}

The main result of this section
 is to show that we can add the conditions that $\Ind(G)$
is a pseudomanifold and that $\Ind(G)$ is a simplicial ball or sphere to the list of equivalent conditions
of \cref{thm.verywellcoveredclass}.  
To prove this result, we need
a graph-theoretic classification of the very
well-covered graphs with the property that $I(G)$ is Cohen-Macaulay. 

\begin{theorem}[\cite{CRT2011}, see also {\cite[Lemma 3.1]{MMCRTY2011}}]\label{thm.verywellcovered-label}
Let $G = (V,E)$ be a very well-covered graph
with $2d$ vertices.  Then $I(G)$ is Cohen-Macaulay 
if and only  if there is a labelling of the vertices 
$V = \{x_1,\ldots,x_d,y_1,\ldots,y_d\}$ such that
\begin{enumerate}
    \item $\{x_1,\ldots,x_d\}$ is a minimal vertex cover of 
  $G$ and $\{y_1,\ldots,y_d\}$ is a maximal independent 
  set of $G$;
    \item $\{x_i,y_i\}$ are edges of $G$ for $i=1,\ldots,d$;
    \item if $\{z_i,x_j\}$ and $\{y_j,x_k\}$ are edges of 
  $G$ for distinct $i,j,k$, then $\{z_i,x_k\} \in E$ for 
  $z_i \in \{x_i,y_i\}$;
    \item if $\{x_i,y_j\}$ is an edge of $G$, then 
  $\{x_i,x_j\} \not\in E$; and 
    \item if $\{x_i,y_j\}$ is an edge of $G$, then 
    $i\leq j$.
\end{enumerate}
\end{theorem}

We can use \cref{thm.verywellcovered-label} 
to describe the
facets of $\Ind(G)$ when $G$ is a very
well-covered graph and $I(G)$ is Cohen-Macaulay.

\begin{corollary}\label{cor.verywellcovered-facets}
  Let $G$ be a very well-covered graph, and suppose
  that $I(G)$ is Cohen-Macaulay.  
  Let $V = \{x_1,\ldots,x_d,y_1,\ldots, y_d\}$ be the 
  labelling of the vertices of $G$ as given in \cref{thm.verywellcovered-label}. Then every 
  facet $F \in \Ind(G)$ has the form
  $$F = \{z_1,z_2,\ldots,z_d\}~~
  \mbox{where $z_i \in \{x_i,y_i\}$ for
  $i =1,\ldots,d$.}$$
\end{corollary}

\begin{proof}
  Let $F \in \Ind(G)$ be a facet.  Since $\{x_i,y_i\}$
  is an edge of $G$ by \cref{thm.verywellcovered-label}~(2),  
  $|F \cap \{x_i,y_i\}| \leq 1$ for $i=1,\ldots, d$ because $F$ is an independent set.
  But since $G$ is very well-covered, every
  maximal independent set, or equivalently, every
  facet, has $d$ elements.  So $|F \cap \{x_i,y_i\}| =1$
  for all $i$.
\end{proof}

We come to our new result about pseudomanifolds. 
Our next statement generalizes our earlier result \cite[Theorem 3.1]{CFHNVT2023} about the independence complexes of whiskered graphs.

\begin{theorem}
\label{thm.verywellcovered-pseudomanifold}
  Let $G$ be a very well-covered graph on $2d$ vertices where $d\geq 1$ and with no
  isolated vertices.  Then $\Ind(G)$ is a pseudomanifold
  if and only if $I(G)$ is Cohen-Macaulay. Furthermore, if $\Ind(G)$
  is a pseudomanifold, then $\Ind(G)$ has
  a boundary if and only
  if $G \neq dK_2$, that is, $G$ is not the graph of $d$ disjoint edges.
\end{theorem}

\begin{proof}
  One direction is clear  by 
  \cref{thm.verywellcoveredclass} and \eqref{implications}: if 
  $\Ind(G)$ is a pseudomanifold,  then $\Ind(G)$ is 
  strongly-connected, 
  and thus $I(G)$ is Cohen-Macaulay.
  
  For the converse, 
  let $G$ be a very well-covered graph with no
  isolated vertices such that $I(G)$
  is Cohen-Macaulay.  
  If $G=dK_2$, then $G$ is a whiskering of a graph with no edges, and so by \cref{thm.mainthm-new}  $\Ind(G)$ 
  is a simplicial sphere, and in particular a pseudomanifold with no boundary (see \eqref{implications}).
  
  Suppose  $G \neq dK_2$ and assume the
  vertices of $G$ have been labeled as in \cref{thm.verywellcovered-label}. Since $G$ is well-covered $\Ind(G)$ is pure of dimension $d-1$, and by \cref{thm.verywellcoveredclass}
  $\Ind(G)$ is strongly-connected.
  Let $H$ be
  any $(d-2)$-dimensional face. By \cref{cor.verywellcovered-facets} $H$ has the form
  $\{z_1,\ldots,\hat{z}_j,\ldots,z_d\}$ with $z_i \in \{x_i,y_i\}$ for all $i$.
  If $F$ is a facet
  that contains $H$, it must have the form
  $H \cup \{x_j\}$ or $H \cup \{y_j\}$.  That is, there are
  at most two facets that contain $H$. Thus 
  $\Ind(G)$ is a pseudomanifold.

  It remains to show the second statement about the boundary when $G \ne dK_s$. With the labeling
  of \cref{thm.verywellcovered-label}, we consider two cases.
  
  First suppose
 that $\{x_1,\ldots,x_d\}$ is an independent set.
 Since $G$ is not a collection
  of disjoint edges, there is pair  $i < j$ such that
  $\{x_i,y_j\}$ is an edge of $G$.  The set $H = 
  \{y_1,\ldots,\hat{y}_i,\ldots,y_d\}$ is a face of dimension $d-2$ of 
  $\Ind(G)$.  It follows from \cref{cor.verywellcovered-facets} that
  $H$ is contained in either the facet $H \cup \{y_i\}$
  or $H \cup \{x_i\}$.  But since $x_i$ is adjacent
  to $y_j \in H$, the only facet containing $H$ is $\{y_1,\ldots,y_d\} \in \Ind(G)$.

  Next, if $\{x_1,\ldots,x_d\}$ is not an independent
  set of $G$, then there is at least one edge $\{x_i,x_j\}$
  among these vertices.  Suppose that $i<j$.  Note
  that $\{y_i,x_j\}$ is not an edge in $G$ by \cref{thm.verywellcovered-label}~(5).
  Consequently, $\{y_i,x_j\}$ belongs to some
  facet
  $H$ of $G$ as every independent set belongs
  to a maximal independent set.
  Then $H \setminus \{y_i\}$ is a face of
  dimension $d-2$ that is contained in $H$. However,
  it is not contained in $(H \setminus \{y_i\}) 
  \cup  \{x_i\}$ since
  $x_i$ and $x_j$ are adjacent.   So $H \setminus \{y_i\}$
  is only contained in the facet $H$.

  These two cases now
  show that $\Ind(G)$ has a boundary.
  \end{proof}

We are now ready to summarize
our new contributions to the
properties of $\Ind(G)$ when $G$ is a
very well-covered graph. In particular,
all the reverse implications in \eqref{implications} that
do not involve the pure property hold 
for this class of simplicial complexes.

\begin{corollary}
\label{cor.verywellcoveredclass}
Let $d \geq 1$ and $G$ be a very well-covered graph on $2d$ vertices with no
  isolated vertices.
Then the following 
conditions are equivalent:
\begin{enumerate}
    \item $\Ind(G)$ is strongly connected;
    \item $I(G)$ is Cohen-Macaulay;
    \item $\Ind(G)$ is a pseudomanifold;
    \item $\Ind(G)$ is a simplicial ball or sphere.
\end{enumerate}
In particular, the following more refined statements are equivalent: 
\begin{itemize}
\item $G=dK_2$ for some $d \geq 1$;
\item $\Ind(G)$ is a pseudomanifold with no boundary;
\item $\Ind(G)$ is a homology sphere;
\item $\Ind(G)$ is a simplicial sphere.
\end{itemize}
\end{corollary}

\begin{proof}
    \cref{thm.verywellcoveredclass,thm.verywellcovered-pseudomanifold} show the
    equivalence of $(1), (2)$ and $(3)$.
    We always have $(4)$ implies $(3)$.
    It remains to show $(3)$ implies $(4)$.  If ${\rm Ind}(G)$ is a 
    pseudomanifold with boundary, we have $I(G)$ is Cohen-Macaulay.
    Since $I(G)$ is Cohen-Macaulay,
    we know that $\Ind(G)$ is shellable
    by \cref{thm.verywellcoveredclass}.
    So $\Ind(G)$ is a shellable pseudomanifold. The rest now follows
    from \cref{thm.verywellcovered-pseudomanifold}, and the the fact that a shellable pseudomanifold is a simplicial sphere if it has no boundary, and is a simplicial ball otherwise
    by \cref{t:sphere-ball}.
\end{proof}

It follows immediately that  if $G$ is a whiskered graph, then ${\rm Ind}(G)$ is a pseudomanifold, this was one of the main tools the authors used in~\cite[Theorem 3.1]{CFHNVT2023}.
We can apply \cref{cor.verywellcoveredclass} to deduce a further corollary.
Recall that a graph $G = (V,E)$ is a {\bf bipartite} graph 
if the vertices can be partitioned as $V = V_1 \cup V_2$ such that
every edge $e \in E$ satisfies $e \cap V_1 \neq \emptyset$ and
$e \cap V_2 \neq \emptyset$.  

\begin{corollary}[{\bf Cohen-Macaulay bipartite graphs}]\label{cor.bipartite+cm=>pseudomanifold}
  Let $G$ be a bipartite graph with no
  isolated vertices. 
  If $I(G)$ is Cohen-Macaulay,
  then $\Ind(G)$ is a pseudomanifold. If $G\neq dK_2$ then
  $\Ind(G)$ is a simplicial ball, and otherwise it is a simplicial sphere.
\end{corollary}

\begin{proof}
  Let $V = V_1 \cup V_2$ be the bipartition of $V$. Since $I(G)$ is 
  Cohen-Macaulay, $\Ind(G)$ is
  pure, and since $V_1$ and $V_2$ are maximal independent sets, we must have $|V_1|=|V_2| = |V|/2$.  So, $G$ must be very well-covered,
  since every maximal independent set must also have cardinality $|V|/2$.
  Now apply \cref{cor.verywellcoveredclass}.
\end{proof}

\section{Applications to Gorenstein rings}

    In~\cite[Theorem II.5.1]{Stanley1996} Stanley shows that if 
    $R/\mathcal{N}(\Delta)$ is a Gorenstein ring, where $R = k[x_1,\dots, x_n]$ and $\mathcal{N}(\Delta)$ is the Stanley-Reisner ideal of $\Delta$, then the simplicial complex $\Delta$ is either a $k$-homology sphere, or a cone over some $k$-homology sphere. In particular, if every variable of $R$ divides a generator of $\mathcal{N}(\Delta)$, then $R/\mathcal{N}(\Delta)$ is Gorenstein if and only if $\Delta$ is a $k$-homology sphere, and in particular, $\Delta$ has non-trivial homology. 

    Throughout this paper, we have shown that two classes of complexes, namely the independence complexes of grafted complexes and very well-covered graphs are always balls, except in the cases where their respective Stanley-Reisner ideals are complete intersections. In a different language, our result implies the following:

    \begin{proposition}\label{p:grafted-very-well-covered}
    Let $I$ be a square-free monomial ideal such that either 
    \begin{enumerate}
        \item the facet complex of $I$ is grafted, or
        \item the facet complex of $I$ is a very-well covered graph.
    \end{enumerate}   
    Then $I$ is Gorenstein if and only if it is a complete intersection.
    \end{proposition}

    \begin{proof}
        Let $\Delta$ be the Stanley-Reisner complex of $I$ and assume $\Delta$ is either the independence complex of a grafted complex or the independence complex of a very well-covered graph. By~\cref{thm.mainthm-new,thm.verywellcovered-label} we know $\Delta$ is not a cone over a face, and in particular by~\cite[Theorem 5.1 Part 2]{Stanley1996}, if $I$ is Gorenstein, $\Delta$  must have non-trivial homology. The result then follows since by~\cref{thm.mainthm-new,thm.verywellcovered-pseudomanifold} $\Delta$ has non-trivial homology exactly when the corresponding Stanley-Reisner ideal is a complete intersection.
    \end{proof}

    \begin{remark} 
        Since polarization preserves the Gorenstein property, and the polarization of every artinian monomial ideal is grafted~(\cref{thm.mainthm-new}), ~\cref{p:grafted-very-well-covered} (1) can be seen as a different proof for the fact that the only Gorenstein artinian monomial ideals are the monomial artinian complete intersections.
    \end{remark}

    In the case of edge ideals, known results from graph theory on well-covered graphs~\cites{P1995,FHN1993} can be used to directly obtain more information on the classification of Gorenstein edge ideals. For the sake of completeness, below we include a proof of a slightly weaker version of~\cite[Remark 3.5]{HT2016} to showcase how the techniques from this paper can be applied to the problem of classifying Gorenstein square-free monomial ideals.
    
\begin{proposition}[{\bf see~{\cite[Remark 3.5]{HT2016}}}]\label{thm.gorensteinmainthm} 
    Suppose that $G$ is a graph such that $I(G)$ is
    Gorenstein.  Then the connected components
    of $G$ are either $K_1$, $K_2$, or graphs
    with girth $\leq 5$.
\end{proposition}

\begin{proof}
    By \cite[Lemma 2.2]{Trung2018}, $G$ is a 
    Gorenstein graph if and only if each connected
    component of $G$ is Gorenstein.  A direct 
    computation shows that $K_1$ and $K_2$
    are Gorenstein.  So, it is suffices
    to show that if $G$ is a connected Gorenstein graph
    with $G \neq K_1$  and $G \neq K_2$, its girth
    is at most five.  
    
    Since $I(G)$ is Gorenstein, $I(G)$
    is Cohen-Macaulay, and 
    thus $G$ is well-covered.
    Suppose $G$ has girth six or greater.
    By~\cite[Corollary 5]{FHN1993}, 
    either $G = C_7$, or $G$ is whiskered.  Since $I(C_7)$ is not Cohen-Macaulay, $C_7$ is 
    also not Gorenstein.  Hence,  
    we only have to show $G$ is not whiskered.
    
    To see that $G$ cannot be a whiskered graph, note that by \cite[Theorem II.5.1]{Stanley1996} the independence complexes of Gorenstein graphs without isolated vertices correspond to flag homology spheres. In particular, the independence complex of $G$ has non-trivial homology. However, the independence complex of a whiskered graph that is not $dK_2$ is a simplicial ball  by~\cref{cor.verywellcoveredclass} and hence has trivial homology. So we conclude $G$ cannot be a whiskered graph.  So the girth of $G$ is
    at most five.
\end{proof}

\subsection*{Acknowledgments}
Part of the work on this project
took place at the Fields Institute in 
Toronto, Canada where some of the authors
participated in the ``Thematic Program in
Commutative Algebra and Applications''.  
Faridi, Holleben, Nicklasson, and Van Tuyl
thank Fields for the financial support they
received.
Cooper's research is supported by NSERC Discovery Grant 
2024-05444.  
Faridi's research is supported by
NSERC Discovery Grant 2023-05929.
Van Tuyl’s research is supported by NSERC Discovery Grant 2024-05299.

\bibliographystyle{plain}
\bibliography{bibliography.bib}

\end{document}